\documentclass[11pt, oneside]{article}   	% use "amsart" instead of "article" for AMSLaTeX format
  \usepackage{geometry}                		% See geometry.pdf to learn the layout options. There are lots.
  \usepackage{amsmath}
  \usepackage{amsthm}
  \usepackage{graphicx}				% Use pdf, png, jpg, or eps§ with pdflatex; use eps in DVI mode
  \usepackage{amssymb}
  \usepackage{authblk}
  
  \newtheorem{theo}{Theorem}[section]
  \newtheorem{lemm}{Lemma}[section]
  \newtheorem{prop}{Proposition}[section]

  \newtheorem{notation}{Notation}[section]
  \newtheorem{assumption}{Assumption}[section]
  \providecommand{\keywords}[1]
{
  \small
  \textbf{\textit{Keywords---}} #1
}

\usepackage{color}

  \def\hoge<#1>{\langle #1 \rangle}
\title{Asymptotic normality of least squares estimators to stochastic differential equations driven by fractional Brownian motions}
\author{Yasutaka Shimizu\thanks{Email:~shimizu@waseda.jp}}
\author{Shohei Nakajima\thanks{Email:~st08m26@akane.waseda.jp}}
\affil{Department of Applied Mathematics, Waseda University, \\ 3-4-1, Okubo, Shinjuku, Tokyo, 169-8555, Japan}
\date{}
\begin{document}
\maketitle
 \section{Introduction}

We will consider the following stochastic differential equation (SDE):
\begin{equation}
 \label{SDE}
  X_t=X_0+\int_0^tb(X_s,\theta_0)ds+\sigma B_t,~~~t\in(0,T],
\end{equation}
where $\{B_t\}_{t\ge 0} $ is a fractional Brownian motion with Hurst index $H\in(1/2,1)$, $\theta_0$ is a parameter that contains a bounded and open convex subset $\Theta\subset\mathbb{R}^d$, $\{b(\cdot,\theta),\theta\in\Theta\}$ is a family of drift coefficients with $b(\cdot,\theta):\mathbb{R}\rightarrow\mathbb{R}$, and $\sigma\in\mathbb{R}$ is assumed to be the known diffusion coefficient.

This study aims to estimate the unknown parameter $\theta_0\in\Theta$. In the case where $H=1/2$—that is, $\{B_t\}_{t\ge 0}$ is a Brownian motion—a maximum likelihood estimator (MLE) is the most widely used when observations of the process are continuous (see Prakasa Rao \cite{Rao}, Kutoyants \cite{Kutoyants}, Liptser, and Shiryaev \cite{Liptser}). When processes are observed at discrete points, the least squares estimator (LSE) is asymptotically equivalent to the MLE.
For the LSE, %asymptotic convergences and
the consistency and asymptotic distributions were proved by Kasonaga \cite{Kasonaga} and Prakasa Rao \cite{Rao1}.

Many authors have studied the parametric %estimation problems of
inference for SDEs driven by fractional Brownian motions based on continuous observation (for example, Brouste and Kleptsyna \cite{Brouste}, Chiba \cite{Chiba}, Hu {\it et al.} %--Nualart--Zhou
\cite{Hu}, Tudor and Viens \cite{Tudor}). In particular, Chiba \cite{Chiba} proposed %suggests
an $M$-estimator for the equation \eqref{SDE} when $1/4<H<1/2$ and proved the asymptotic normality and moment convergence.

Neuenkirch and Tindel \cite{Neuenkirch} studied estimators based on discrete samples %time data
from equation \eqref{SDE}  and showed the consistency of LSE when $1/2<H<1$. However, to the best of our knowledge, asymptotic normality of estimators based on discrete observations has yet to be analyzed. The purpose of this study is to show the asymptotic normality of the LSE. In contrast, Liu {\it et al.} \cite{Liu} proved the LAN property for the equation \eqref{SDE},
and the optimality of the asymptotic variance of estimator is already known. %But
However, our results do not imply %mean
that LSE is asymptotically efficient. Constructing an asymptotic efficient estimator is a task for future research. The main tools %of
in the proof %is
of the asymptotic normality are the central limit theorem of the 2-Hermite power variation and the ergodic theorem to sums of the increments weighted by fractional Brownian motions.

\section{Main results}
We assume that $\{X_t\}_{t\ge 0}$ is observed at points $\{t_k:~0\le k\le n\}$ and take equally spaced observations %time
 with $t_{k+1}-t_k:=h_n$.
Define the least squares type procedure,
\begin{equation*}
  Q_n(\theta):=\frac{1}{nh_n^2}\sum_{k=1}^{n}\left(\left(X_{t_k}-X_{t_{k-1}}-h_nb(X_{t_{k-1}},\theta)\right)^2-\sigma^2h_n^{2H}\right),
\end{equation*}
and the LSE for the true $\theta_0$ is defined as
\begin{equation}
  \theta_n:=\arg\min_{\theta\in\Theta}|Q_n(\theta)|. \label{LSE}
\end{equation}
Throughout the paper %thesis,
we will use the following notations:
\begin{notation}
%We shall use the notation
For any $a,b\ge 0$, the symbol $a\lesssim b$ means that there exists a universal constant $C>0$ such that $a\le Cb$. When $C$ depends explicitly on a specific quantity, we indicate it explicitly through the paper.
\end{notation}
Let us state the polynomial growth assumptions to drift coefficients $b$, ensuring ergodic properties for process $X$.
\begin{assumption}
  \label{(A1)}
The function $b$ in \eqref{SDE} is of $C^{1,3}(\mathbb{R}\times\Theta)$-class
such that, for every $x,y\in\mathbb{R}$ and $\theta\in\Theta$,
\begin{equation*}
  \left(b(x,\theta)-b(y,\theta)\right)(x-y)\le -c|x-y|^2.
\end{equation*}
and the following growth conditions hold true
\begin{equation*}
  \begin{aligned}
    &|b(x,\theta)|\le c(1+|x|^N),~|\partial_x b(x,\theta)|\le c(1+|x|^N),~|\nabla_\theta b(x,\theta)|\le c(1+|x|^N)\\
    &|\nabla_\theta\partial_x b(x,\theta)|\le c(1+|x|^N),~|\nabla^2_\theta b(x,\theta)|\le c(1+|x|^N),
  \end{aligned}
\end{equation*}
for some constants $c>0,~N\in\mathbb{N}$.
  \end{assumption}
  \begin{assumption}
    \label{(A2)}
    There exists a function $U\in C^{2,3}(\mathbb{R}\times\Theta)$ such that
    \begin{equation*}
      \partial_x U(x,\theta)=b(x,\theta),
    \end{equation*}
    for every $x\in\mathbb{R},~\theta\in\Theta$.
  \end{assumption}
  We impose the condition on the size of the sampling %discretization
  step, which is required to control the contribution of fractional Brownian motions.
  \begin{assumption}
    \label{(A3)}
    $h_n=\kappa n^{-\alpha}$ with $0<\alpha<\min\{\frac{1}{4(1-H)},1\}$ and $\kappa>0$.
  \end{assumption}
      Under Assumption %condition
      \ref{(A1)}, the solution to \eqref{SDE} converges %for $T\rightarrow\infty$
      to a stationary and ergodic stochastic process $(\bar{X_t},~t\ge 0)$ as $T\to \infty$ (see Theorem \ref{ergodic}).
\begin{assumption}
  \label{(A4)}
  For every $\theta_0\in\Theta$,
  \begin{equation*}
    E\left|b(\bar{X}_0,\theta_0)\right|^2=  E\left|b(\bar{X}_0,\theta)\right|^2,
  \end{equation*}
 % if and only if
implies that $\theta=\theta_0$.
\end{assumption}
\begin{assumption}
  \label{(A5)}
  The matrix
  \begin{equation*}
    I(\theta_0)=E\left[\left(\nabla_\theta{b}(\bar{X},\theta_0)b(\bar{X},\theta_0)\right)\right]^{\otimes2},
  \end{equation*}
  is positive definite.
\end{assumption}
The consistency of LSE \eqref{LSE} was given by Neuenkirch and Tindel \cite{Neuenkirch}.
\begin{theo}[Neuenkirch and Tindel \cite{Neuenkirch}]
  \label{consistency}
  Under Assumptions \ref{(A1)}-\ref{(A4)},
  the LSE $\theta_n$ is strongly consistent with $\theta_0$:
 \[
 \theta_n\rightarrow \theta_0\quad a.s.,\quad n\to \infty.
 \]
\end{theo}
To state the main results, we further make some notations: let
\begin{equation*}
  \begin{aligned}
    \tau_n^H= \begin{cases}
    &\sqrt{n}h_n^{2-2H},~~~~~~~~~~~H\in(1/2,3/4)\\
    &\sqrt{\frac{n}{\log(n)}}h_n^{2-2H},~~~~~~H=3/4\\
    &(nh_n)^{2-2H},~~~~~~~~~~H\in(3/4,1).
  \end{cases}
  \end{aligned}
\end{equation*}
\begin{theo}
  \label{main}
%In addition to
Suppose the same assumptions as in Theorem 1, and that $\frac{1}{2}<\alpha<\min\{\frac{1}{4(1-H)},1\}$, Assumption \ref{(A5)} holds true. Then, if $H\in(1/2,3/4]$, there exists a constant $c_H>0$ such that
  \begin{equation*}
    \tau_n^H(\theta_n-\theta_0)\xrightarrow{d}N\left(0,\frac{\sigma^2 c_H^2}{4}\left(I(\theta_0)\right)^{-1}\right),
  \end{equation*}
as $n\rightarrow\infty$. Moreover, if $H\in(3/4,1)$,
  \begin{equation*}
    \tau_n^H(\theta_n-\theta_0)\xrightarrow{d} \frac{\sigma \left(I(\theta_0)\right)^{-1}E\left[\nabla_\theta{b}(\bar{X},\theta_0)b(\bar{X},\theta_0)\right]}{2}Z,
  \end{equation*}
  as $n\rightarrow\infty$ where %$Z$ is the Hermite random variable {\red given below,} 
  $\bar{X}$ is the ergodic limit of the process $X$, and Z is a {\it Hermite random variable } that will be defined in the next section.
\end{theo}

%\section{Preliminaries} %3.1節しかないので，subsectionは削除
\section{Ergodic theorem}
In this section, we shall describe the %prepare
results of the ergodic theorem in equation \eqref{SDE}. We will work on the canonical probability space $(\Omega,\mathcal{F},P)$, %, which is
where $\Omega=C_0(\mathbb{R})$ is equipped with the topology of the compact convergence, $\mathcal{F}$ is the corresponding Borel $\sigma$-algebra, and $P$ is the distribution of the fractional Brownian motion. We define the shift operator  $\theta_t:\Omega\rightarrow\Omega$ for each $t\in\mathbb{R}$ and $\omega\in\Omega$ as
\begin{equation*}
  \theta_t\omega(\cdot)=\omega(t+\cdot)-\omega(t).
\end{equation*}
The shifted process $(B_s(\theta_t\cdot))_{s\in\mathbb{R}}$ is a %an
1--dimensional fractional Brownian motion, and, for any integrable random variable $F:\Omega\rightarrow\mathbb{R}$ and any $\omega\in\Omega$, we have
\begin{equation*}
  \lim_{T\rightarrow\infty}\int_0^TF(\theta_t(\omega))dt=E[F].
\end{equation*}
%for $\omega\in\Omega$.

Let us recall the definition of H\"older continuous functions:
\begin{equation*}
  C^\lambda([a,b];\mathbb{R})=\left\{f:[a,b]\rightarrow\mathbb{R};\|f\|_{\lambda;[a,b]}<\infty,\right\}
\end{equation*}
where
\begin{equation*}
  \|f\|_{\lambda;[a,b]}=\sup_{s,t\in[a,b]}\frac{\left|f(t)-f(s)\right|}{|t-s|^\lambda}.
\end{equation*}

We state the existence and uniqueness of the ergodic limit for $X$ and introduce some limit theorems borrowed from Garrido {\it et al.} \cite{Garrido}.
\begin{prop}
  Under Assumption \ref{(A1)}, for any $\theta\in\Theta$, the equation \eqref{SDE} has a unique solution $X\in C^{\lambda}(\mathbb{R}_+;\mathbb{R})$ for all $\lambda<H$. In addition, there exists a random variable $\bar{X}:\Omega\rightarrow\mathbb{R}$ such that
  \begin{equation*}
    \lim_{t\rightarrow\infty}\left|X_t(\omega)-\bar{X}(\theta_t\omega)\right|=0,
  \end{equation*}
  for all $\omega\in\Omega$.
\end{prop}
An ergodic theorem for discrete sampling is found in Lemma 3.3 in Neuenkirch and Tindel \cite{Neuenkirch}.
\begin{lemm}
    \label{ergodic}
    Let $f\in C^{1,1}(\mathbb{R}\times\Theta)$ such that
    \begin{equation*}
      \begin{aligned}
        &|f(x,\theta)|\le C(1+|x|^N),~|\partial_x f(x,\theta)|\le C(1+|x|^N),~|\nabla_\theta f(x,\theta)|\le C(1+|x|^N),
      \end{aligned}
    \end{equation*}
    for some $c>0, N\in\mathbb{N}$. Then
    \begin{equation*}
      \sup_{\theta\in\Theta}\left|\frac{1}{n}\sum_{k=1}^nf(X_{t_{k-1}},\theta)-Ef(\bar{X},\theta)\right|\rightarrow 0,~~~a.s.
    \end{equation*}
    In addition, assume that there exists a function $U\in C^{2,1}(\mathbb{R}\times\Theta)$ such that
    \begin{equation*}
      \partial_xU(x,\theta)=f(x,\theta),~ x\in\mathbb{R},~\theta\in\Theta.
    \end{equation*}
    Then
\begin{equation*}
    \sup_{\theta\in\Theta}\left|\frac{1}{nh_n}\sum_{k=1}^nf\left(X_{t_{k-1}},\theta\right)(B_{t_k}-B_{t_{k-1}})+E\left[b(\bar{X},\theta_0)f(\bar{X},\theta)\right]\right|\rightarrow 0~~~a.s.
\end{equation*}
as $n\rightarrow\infty$.
  \end{lemm}
We prepare some estimate results for the pth moment of the solution to \eqref{SDE}. To support these results, we refer to \cite{Garrido} and \cite{Neuenkirch}.
\begin{prop}
  \label{estimate}
  Under Assumption \ref{(A1)}, for any $\theta\in\Theta$ and $p\ge 1$, there exist constants $c_p,k_p>0$ such that
  \begin{equation*}
    E|X_t|^p\le c_p,~~~~~E|X_t-X_s|^p\le k_p|t-s|^{pH}
  \end{equation*}
  for all $s,t\ge 0$.
\end{prop}
We need the following convergence results for one--dimensional fractional Brownian motion (cf. Theorem 7.4.1 in \cite{Nourdin}).
\begin{prop}
  \label{fractional Brownian motion Convergence}
  We set
  \begin{equation*}
    c_H:=\frac{1}{2}\sum_{v\in\mathbb{Z}}\left(|v+1|^{2H}+|v-1|^{2H}-2|v|^{2H}\right)^2<\infty.
  \end{equation*}
  Then for $0<H<3/4$
  \begin{equation*}
    \frac{1}{\sqrt{n}c_H}\sum_{k=1}^n[(B_k-B_{k-1})^2-1]\xrightarrow{d}N(0,1),
  \end{equation*}
  while for $H=3/4$ it holds
  \begin{equation*}
    \frac{1}{\sqrt{n\log(n)}c_{3/4}}\sum_{k=1}^n[(B_k-B_{k-1})^2-1]\xrightarrow{d}N(0,1).
  \end{equation*}
  Finally, for $3/4<H<1$
  \begin{equation*}
    \frac{1}{n^{2H-1}}\sum_{k=1}^n[(B_k-B_{k-1})^2-1]
  \end{equation*}
  converges in $L^2(\Omega)$ to some random variable $Z$, which belongs to the Wiener chaos of $B$ with order 2. The random variable $Z$ is called a {\it Hermite random variable}.
\end{prop}
\section{Proofs}
 Let $\mathcal{F}_{t_k}:=\sigma(B_s,~0\le s\le t_k)$ and note that
\begin{equation*}
  \nabla_\theta Q_n(\theta)=\frac{1}{nh_n}\sum_{k=1}^{n}\left(X_{t_k}-X_{t_{k-1}}-h_nb(X_{t_{k-1}},\theta)\right)\nabla_\theta b(X_{t_{k-1}},\theta).
\end{equation*}
Since $\theta_n$ minimizes to $Q_n(\theta)^2$, we obtain that
\begin{equation*}
  Q_n(\theta_n)\nabla_\theta Q_n(\theta_n)=0.
\end{equation*}
To solve the above equation, we prepare the following lemma.
\begin{lemm}
  \label{calculus}
  Define $\zeta_n=\tau_n^H(\theta_0-\theta_n)$. Then, under assumptions of Theorem \ref{main},
  \begin{equation*}
    \left \{
    \begin{aligned}
      \tau_n^HQ_n(\theta_n)&=\tau_n^H\frac{1}{nh_n^2}\sum_{k=1}^n\left[\sigma^2(B_{t_k}-B_{t_{k-1}})^2-\sigma^2h_n^{2H}\right]+\frac{2\sigma}{nh_n}\sum_{k=1}^n(B_{t_k}-B_{t_{k-1}})\nabla_\theta b(X_{t_{k-1}},\theta_0)\zeta_n^T\\
      &~~~~~+\zeta_no_p(1)+o_p(1)\\
      \nabla_\theta Q_n(\theta_n)&=\frac{1}{nh_n}\sum_{k=1}^{n}\sigma(B_{t_k}-B_{t_{k-1}})\nabla_\theta b(X_{t_{k-1}},\theta_n)+o_p(1).\\
    \end{aligned}\right.
  \end{equation*}
\end{lemm}
\begin{proof}
  By using \eqref{SDE}, we have
  \begin{equation*}
    \begin{aligned}
      Q_n(\theta_n)&=\frac{1}{nh_n^2}\sum_{k=1}^n\left(\left(X_{t_k}-X_{t_{k-1}}-h_nb(X_{t_{k-1}},\theta_n)\right)^2-\sigma^2h_n^{2H}\right)\\
      &=\frac{1}{nh_n^2}\sum_{k=1}^n\left\{\left(\int_{t_{k-1}}^{t_k}\left(b(X_s,\theta_0)-b(X_{t_{k-1}},\theta_n)\right)ds+\sigma(B_{t_k}-B_{t_{k-1}})\right)^2-\sigma^2h_n^{2H}\right\}\\
      &=\frac{1}{nh_n^2}\sum_{k=1}^n\Biggl\{\left(\int_{t_{k-1}}^{t_k}\left(b(X_s,\theta_0)-b(X_{t_{k-1}},\theta_n)\right)ds\right)^2+\sigma^2(B_{t_k}-B_{t_{k-1}})^2\\
      &~~~~~~~~~~~~~~~~~~~~~~~~~~~~+2\sigma(B_{t_k}-B_{t_{k-1}})\int_{t_{k-1}}^{t_k}\left(b(X_s,\theta_0)-b(X_{t_{k-1}},\theta_n)\right)ds-\sigma^2h_n^{2H}\Biggl\}.
    \end{aligned}
  \end{equation*}
  For the first term, we can calculate that
  \begin{equation*}
    \begin{aligned}
  &\left(\int_{t_{k-1}}^{t_k}\left(b(X_s,\theta_0)-b(X_{t_{k-1}},\theta_n)\right)ds\right)^2\\
  &=\left(\int_{t_{k-1}}^{t_k}\left(b(X_s,\theta_0)-b(X_{t_{k-1}},\theta_0)\right)ds+\int_{t_{k-1}}^{t_k}\left(b(X_{t_{k-1}},\theta_0)-b(X_{t_{k-1}},\theta_n)\right)ds\right)^2\\
  &=\left(\int_{t_{k-1}}^{t_k}\left(b(X_s,\theta_0)-b(X_{t_{k-1}},\theta_0)\right)ds\right)^2+\left(\int_{t_{k-1}}^{t_k}\left(b(X_{t_{k-1}},\theta_0)-b(X_{t_{k-1}},\theta_n)\right)ds\right)^2\\
  &~~~~~~~~~~~~~~~~~~~+2\left(\int_{t_{k-1}}^{t_k}\left(b(X_s,\theta_0)-b(X_{t_{k-1}},\theta_0)\right)ds\right)\left(\int_{t_{k-1}}^{t_k}\left(b(X_{t_{k-1}},\theta_0)-b(X_{t_{k-1}},\theta_n)\right)ds\right)
    \end{aligned}
  \end{equation*}
Through Minkowski's inequality and Proposition \ref{estimate} we can estimate that
\begin{equation*}
  \begin{aligned}
    &\frac{1}{nh_n^2}\sum_{k=1}^n\left\|\int_{t_{k-1}}^{t_k}\left(b(X_s,\theta_0)-b(X_{t_{k-1}},\theta_0)\right)ds\right\|_{L^2(\Omega)}^2\\
    &\le \frac{1}{nh_n^2}\sum_{k=1}^n\left(\int_{t_{k-1}}^{t_k}\left\|b(X_s,\theta_0)-b(X_{t_{k-1}},\theta_0)\right\|_{L^2(\Omega)}ds\right)^2\lesssim h_n^{2H}.
  \end{aligned}
\end{equation*}
Therefore, under the assumption $nh_n^2\rightarrow 0$ as $n\rightarrow\infty$, we obtain
\begin{equation*}
  \tau_n^H\frac{1}{nh_n^2}\sum_{k=1}^n\left(\int_{t_{k-1}}^{t_k}\left(b(X_s,\theta_0)-b(X_{t_{k-1}},\theta_0)\right)ds\right)^2\xrightarrow{L^1}0,
\end{equation*}
as $n\rightarrow\infty.$ Applying Taylor's formula, we have
\begin{equation*}
  \begin{aligned}
    &\left(\int_{t_{k-1}}^{t_k}\left(b(X_s,\theta_0)-b(X_{t_{k-1}},\theta_0)\right)ds\right)\left(\int_{t_{k-1}}^{t_k}\left(b(X_{t_{k-1}},\theta_0)-b(X_{t_{k-1}},\theta_n)\right)ds\right)\\
    &=h_n\left(\int_{t_{k-1}}^{t_k}\left(b(X_s,\theta_0)-b(X_{t_{k-1}},\theta_0)\right)ds\right)\nabla_\theta b(X_{t_{k-1}},\tilde{\theta}_n)(\theta_0-\theta_n)^T,
  \end{aligned}
\end{equation*}
where $\tilde{\theta}_n:=\theta_0+\beta_n(\theta_n-\theta_0),~0<\beta_n<1$.
With Cauchy--Schwartz and Minkowski's inequalities, Proposition \ref{estimate}, and (A1), we obtain 
\begin{equation*}
 \begin{aligned}
   &\frac{1}{nh_n^2}\sum_{k=1}^nE\left|h_n\left(\int_{t_{k-1}}^{t_k}\left(b(X_s,\theta_0)-b(X_{t_{k-1}},\theta_0)\right)ds\right)\nabla_\theta b(X_{t_{k-1}},\tilde{\theta}_n)\right|\\
   &\lesssim \frac{1}{nh_n}\sum_{k=1}^n\left\|\int_{t_{k-1}}^{t_k}\left(b(X_s,\theta_0)-b(X_{t_{k-1}},\theta_0)\right)ds\right\|_{L^2(\Omega)}\left\|1+|X_{t_{k-1}}|^N\right\|_{L^2(\Omega)}\\
   &\lesssim\frac{1}{nh_n}\sum_{k=1}^n\int_{t_{k-1}}^{t_k}\left\|b(X_s,\theta_0)-b(X_{t_{k-1}},\theta_0)\right\|_{L^2(\Omega)}ds
   \lesssim h_n^{H}.
 \end{aligned}
\end{equation*}
Therefore
\begin{equation*}
  \tau_n^H\frac{1}{nh_n^2}\sum_{k=1}^n\left(\int_{t_{k-1}}^{t_k}\left(b(X_s,\theta_0)-b(X_{t_{k-1}},\theta_0)\right)ds\right)\left(\int_{t_{k-1}}^{t_k}\left(b(X_{t_{k-1}},\theta_0)-b(X_{t_{k-1}},\theta_n)\right)ds\right)\xrightarrow{L^1}0,
\end{equation*}
as $n\rightarrow\infty$. With the Taylor's formula again, we have
\begin{equation*}
  \begin{aligned}
    \tau_n^H\frac{1}{nh_n^2}\sum_{k=1}^n\left(\int_{t_{k-1}}^{t_k}\left(b(X_{t_{k-1}},\theta_0)-b(X_{t_{k-1}},\theta_n)\right)ds\right)^2&=\frac{\tau_n^H}{n}\sum_{k=1}^n\Biggl(\nabla_\theta b(X_{t_{k-1}},\theta_0)(\theta_0-\theta_n)^T\\
    &\hspace{50pt}+(\theta_0-\theta_n)\nabla_\theta^2b(X_{t_{k-1}},\tilde{\theta}_n)(\theta_0-\theta_n)^T\Biggl)^2.
  \end{aligned}
\end{equation*}
Through Theorem \ref{consistency} and Proposition \ref{ergodic}, we can obtain
\begin{equation*}
  \tau_n^H\frac{1}{nh_n^2}\sum_{k=1}^n\left(\int_{t_{k-1}}^{t_k}\left(b(X_s,\theta_0)-b(X_{t_{k-1}},\theta_n)\right)ds\right)^2=o_P(1)+\zeta_no_p(1).
\end{equation*}
Let us consider the cross term. At first, we decompose as follows
\begin{equation*}
  \begin{aligned}
      &\tau_n^H\frac{2\sigma}{nh_n^2}\sum_{k=1}^n(B_{t_k}-B_{t_{k-1}})\int_{t_{k-1}}^{t_k}\left(b(X_s,\theta_0)-b(X_{t_{k-1}},\theta_n)\right)ds\\
      &=\tau_n^H\frac{2\sigma}{nh_n^2}\sum_{k=1}^n(B_{t_k}-B_{t_{k-1}})\int_{t_{k-1}}^{t_k}\left(b(X_s,\theta_0)-b(X_{t_{k-1}},\theta_0)\right)ds\\
      &~~~~~~~~~~~~~~~~~+\tau_n^H\frac{2\sigma}{nh_n^2}\sum_{k=1}^n(B_{t_k}-B_{t_{k-1}})\int_{t_{k-1}}^{t_k}\left(b(X_{t_{k-1}},\theta_0)-b(X_{t_{k-1}},\theta_n)\right)ds.
  \end{aligned}
\end{equation*}
From Cauchy—Schwartz and Minkowski's inequalities and Proposition \ref{estimate}, we obtain
\begin{equation*}
  \begin{aligned}
    &E\left|(B_{t_k}-B_{t_{k-1}})\int_{t_{k-1}}^{t_k}\left(b(X_s,\theta_0)-b(X_{t_{k-1}},\theta_0)\right)ds\right|\\
    &\le\left\|B_{t_k}-B_{t_{k-1}}\right\|_{L^2(\Omega)}\int_{t_{k-1}}^{t_k}\left\|b(X_s,\theta_0)-b(X_{t_{k-1}},\theta_0)\right\|_{L^2(\Omega)}ds\lesssim h_n^{2H+1}.
  \end{aligned}
\end{equation*}
Thus,
\begin{equation*}
     \tau_n^H\frac{2\sigma}{nh_n^2}\sum_{k=1}^n(B_{t_k}-B_{t_{k-1}})\int_{t_{k-1}}^{t_k}\left(b(X_s,\theta_0)-b(X_{t_{k-1}},\theta_0)\right)ds=o_p(1),
\end{equation*},
and we conclude that
\begin{equation*}
  \begin{aligned}
      \tau_n^HQ_n(\theta_n)&=\tau_n^H\frac{1}{nh_n^2}\sum_{k=1}^n\left[\sigma^2(B_{t_k}-B_{t_{k-1}})^2-\sigma^2h_n^{2H}\right]+\frac{2\sigma}{nh_n}\sum_{k=1}^n(B_{t_k}-B_{t_{k-1}})\nabla_\theta b(X_{t_{k-1}},\theta_0)\zeta_n^T\\
      &~~~~~+\zeta_no_p(1)+o_p(1).
  \end{aligned}
\end{equation*}
We will support the second equality. Using \eqref{SDE}, we have
\begin{equation*}
  \begin{aligned}
    \nabla_\theta Q_n(\theta_n)&=\frac{1}{nh_n}\sum_{k=1}^{n}\left(X_{t_k}-X_{t_{k-1}}-h_nb(X_{t_{k-1}},\theta_n)\right)\nabla_\theta b(X_{t_{k-1}},\theta_n)\\
    &=\frac{1}{nh_n}\sum_{k=1}^{n}\left(\int_{t_{k-1}}^{t_k}\left(b(X_s,\theta_0)-b(X_{t_{k-1}},\theta_n)\right)ds+\sigma(B_{t_k}-B_{t_{k-1}})\right)\nabla_\theta b(X_{t_{k-1}},\theta_n).
  \end{aligned}
\end{equation*}
Let us now apply Taylor's formula to obtain that
\begin{equation*}
  \begin{aligned}
      \frac{1}{nh_n}\sum_{k=1}^{n}&\left(\int_{t_{k-1}}^{t_k}\left(b(X_s,\theta_0)-b(X_{t_{k-1}},\theta_n)\right)ds\right)\nabla_\theta b(X_{t_{k-1}},\theta_n)\\
      &=\frac{1}{nh_n}\sum_{k=1}^{n}\left(\int_{t_{k-1}}^{t_k}\left(b(X_s,\theta_0)-b(X_{t_{k-1}},\theta_0)\right)ds\right)\nabla_\theta b(X_{t_{k-1}},\theta_n)\\
      &~~~~~~~+\frac{1}{nh_n}\sum_{k=1}^{n}\left(\int_{t_{k-1}}^{t_k}\left(b(X_{t_{k-1}},\theta_0)-b(X_{t_{k-1}},\theta_n)\right)ds\right)\nabla_\theta b(X_{t_{k-1}},\theta_n)\\
      &=\frac{1}{n}\sum_{k=1}^{n}\nabla_\theta b(X_{t_{k-1}},\theta_0)(\theta_0-\theta_n)^T\nabla_\theta b(X_{t_{k-1}},\theta_0)+o_p(1).
  \end{aligned}
\end{equation*},
and we complete the proof.
\end{proof}
\begin{proof}[Proof of Theorem \ref{main}]
  Since the relationship $\tau_n^HQ_n(\theta_n)\nabla_\theta{Q}_n(\theta_n)=0$ holds, we have
  \begin{equation*}
    \begin{aligned}
      &\left(\tau_n^H\frac{1}{nh_n^2}\sum_{k=1}^n\left[\sigma^2(B_{t_k}-B_{t_{k-1}})^2-\sigma^2h_n^{2H}\right]+\left(\frac{2\sigma}{nh_n}\sum_{k=1}^n(B_{t_k}-B_{t_{k-1}})\nabla_\theta b(X_{t_{k-1}},\theta_0)+o_p(1)\right)\zeta_n^T+o_p(1)\right)\\
      &\hspace{50pt}\times\left(\frac{1}{nh_n}\sum_{k=1}^{n}\sigma(B_{t_k}-B_{t_{k-1}})\nabla_\theta b(X_{t_{k-1}},\theta_n)+o_p(1)\right)=0.
    \end{aligned}
  \end{equation*}
  By using Lemma \ref{ergodic} and Proposition \ref{fractional Brownian motion convergence }, we conclude the statement of Theorem \ref{main}.
\end{proof}

\paragraph{Acknowledgments}
The first author was partially supported by JSPS KAKENHI Grant Numbers JP21K03358 and JST CREST JPMJCR14D7, Japan.


\begin{thebibliography}{99}
   \bibitem{Brouste}{\rm Brouste, A. and Kleptsyna, M. (2010). Asymptotic properties of MLE for partially observed fractional diffusion system,} Stat. Infer. Stochastic Process., {\rm \textbf{13} (1), 1--13.}
   \bibitem{Chiba}{\rm Chiba, K. (2020). An M-estimator for stochastic differential equations driven by fractional Brownian motion with small Hurst parameter,} Statistical Inference for Stochastic Processes, {\rm \textbf{23}, 319--353.}
   \bibitem{Garrido}{\rm Garrido--Atienza, M., Kloeden, P. and Neuenkirch, A. (2009). Discretization of stationary solutions of stochastic systems driven by fractional Brownian motion,} Appl. Math. Optim, {\rm \textbf{60} (2), 151--172.}
   \bibitem{Hu}{\rm Hu, Y., Nualart, D. and Zhou, H. (2019). Drift parameter estimation for nonlinear stochastic differential equations driven by fractional Brownian motion,} Stochastics, {\rm \textbf{91}, 1--25.}
   \bibitem{Kasonaga}{\rm Kasonga, R.A. (1988). The consistency of a nonlinear least squares estimator for diffusion processes,} Stochastic Process. Appl. {\rm \textbf{30}, 263--275.}
   \bibitem{Kutoyants}{\rm Kutoyants, Y.A. (2004).} {\it Statistical Inference for Ergodic Diffusion Processes,} {\rm Springer-Verlag, London, Berlin, Heidelberg.}
   \bibitem{Liptser}{\rm Liptser, R.S. and Shiryayev, A.N. (2001).} {\it Statistics of Random Processes: II Applications, Second Edition, Applications of Mathematics,} {\rm Springer-Verlag, Berlin, Heidelberg, New York.}
   \bibitem{Liu}{\rm Liu, Y., Nualart, E. and Tindel, S. (2019). LAN property for stochastic differential equations with additive fractional,} Stochastic Process. Appl., {\rm \textbf{129}, 2880--2902.}
   \bibitem{Neuenkirch}{\rm Neuenkirch, A. and Tindel, S. (2014). A least square--type procedure for parameter estimation in stochastic differential equations with additive fractional noise,} Stat. Inference Stoch. Process. {\rm  \textbf{17} (1), 99--120.}
   \bibitem{Nourdin}{\rm Nourdin, I. and Peccati, G. (2012).} {\it Normal Approximations with Malliavin Calculus: From Stein's Method to Universality,} {\rm Cambridge University Press, Cambridge.}
   \bibitem{Rao1}{\rm Prakasa Rao, B.L.S. (1983). Asymptotic theory for nonlinear least squares estimator for diffusion processes,} Math. Operations forschung Statist Ser. Statist. {\rm \textbf{14}, 195--209.}
   \bibitem{Rao}{\rm Prakasa Rao, B.L.S. (1999).} {\it Statistical Inference for Diffusion Type Processes,} {\rm Arnold, London, Oxford University Press, New York.}
   \bibitem{Tudor}{\rm Tudor, C. and Viens, F. (2007). Statistical aspects of the fractional stochastic calculus,} Ann. Stat. {\rm \textbf{35}, 1183--1212.}
 \end{thebibliography}
\end{document}